\documentclass[12pt, reqno]{amsart}
\usepackage{fullpage}
\usepackage{amsmath, amsthm, amssymb}
\usepackage{mathrsfs}
\usepackage{enumerate}
\usepackage{tikz}
\tikzstyle{v}=[circle, draw, fill=white,
                        inner sep=0pt, minimum width=4pt]
\tikzset{>=stealth}

\newcommand{\vc}[1]{\rule[-3ex]{0pt}{0pt}\rule[4ex]{0pt}{0pt} \raisebox{\dimexpr-.5\height+.5\ht\strutbox\relax}{#1}}

\newtheorem{thm}{Theorem}[section]

\newtheorem{prop}[thm]{Proposition}
\newtheorem{cor}[thm]{Corollary}
\newtheorem{lemma}[thm]{Lemma}

\theoremstyle{definition}

\newtheorem*{defn}{Definition}
\newtheorem{ex}[thm]{Example}

\theoremstyle{remark}
\newtheorem{rmk}[thm]{Remark}

\newcommand{\FK}{\mathcal E}
\newcommand{\RR}{\mathbf R}
\newcommand{\QQ}{\mathbf Q}

\newcommand{\Hilb}{\mathcal H}
\renewcommand{\S}{\mathfrak S}
\DeclareMathOperator{\supp}{supp}

\begin{document}
\title{On the commutative quotient of Fomin-Kirillov algebras}
\author{Ricky Ini Liu}
\address{Department of Mathematics, North Carolina State University, Raleigh, NC }
\email{riliu@ncsu.edu}

\begin{abstract}
The Fomin-Kirillov algebra $\FK_n$ is a noncommutative algebra with a generator for each edge in the complete graph on $n$ vertices. For any graph $G$ on $n$ vertices, let $\FK_G$ be the subalgebra of $\FK_n$ generated by the edges in $G$.
We show that the commutative quotient of $\FK_G$ is isomorphic to the Orlik-Terao algebra of $G$. As a consequence, the Hilbert series of this quotient is given by $(-t)^n \chi_G(-t^{-1})$, where $\chi_G$ is the chromatic polynomial of $G$. We also give a reduction algorithm for the graded components of $\FK_G$ that do not vanish in the commutative quotient and show that their structure is described by the combinatorics of noncrossing forests.
\end{abstract}

\maketitle

\section{Introduction}

Fomin and Kirillov \cite{FK} introduced a noncommutative algebra $\FK_n$ for the purpose of understanding the generalized Littlewood-Richardson problem of computing intersection numbers in the flag variety. Since then, this algebra and its generalizations have been studied extensively elsewhere: see, for instance, \cite{Bazlov, BLM, FP, Kirillov, KirillovMaeno, Lenart, LenartMaeno, Majid, MPP, MS, P, Vendramin}. Unfortunately, many key facts about $\FK_n$, such as its Hilbert series, remain unknown for most values of $n$.  In order to further the study of $\FK_n$, the authors of \cite{BLM} describe a subalgebra of $\FK_G \subset \FK_n$ for any graph $G$ on $n$ vertices and discuss its properties. 

Fomin and Kirillov \cite{FK} attribute to Varchenko the observation that the commutative quotient of $\FK_n$ has dimension $n!$ and Hilbert series $(1+t)(1+2t) \cdots (1+(n-1)t)$. In this paper, we will show that an analogous result holds for the commutative quotient $\FK_G^{ab}$ of $\FK_G$. Specifically, we will show that $\FK_G^{ab}$ is isomorphic to the Orlik-Terao algebra $U_G$ (defined in \cite{OT}), which is known to have Hilbert series $(-t)^n \chi_G(-t^{-1})$. This resolves a conjecture of Kirillov \cite{Kirillov}.

We also discuss the structure of the graded components of $\FK_G$ that do not vanish in the commutative quotient. This can be thought of as giving a sort of noncommutative analogue of the Orlik-Terao algebra. We show that one can describe a basis for these components in terms of certain noncrossing forests. We also give a reduction algorithm for these components and show that this reduction gives a unique expression of any element in terms of basis elements (independent of the choices made during the reduction). This reduction is similar to the reductions in subdivision algebras given by M\'esz\'aros in \cite{Meszaros}.

We begin in Section 2 with the proof that $\FK_G^{ab}$ is isomorphic to the Orlik-Terao algebra $U_G$. In Section 3, we discuss the combinatorics of noncrossing trees and their relationship to the structure of certain graded components of $\FK_G$. Finally, we conclude in Section 4 with some brief final remarks and observations.

\section{Fomin-Kirillov algebras}

We begin with some preliminaries about the Fomin-Kirillov algebras $\FK_G$. For more information, see \cite{BLM, FK}.

\begin{defn}
	The \emph{Fomin-Kirillov algebra} $\FK_n$ is the quadratic algebra (say, over $\QQ$) with generators
	$x_{ij}=-x_{ji}$ for $1 \leq i < j \leq n$ with the following relations:
	\begin{itemize}
		\item $x_{ij}^2 = 0$ for distinct $i,j$;
		\item $x_{ij}x_{kl} = x_{kl}x_{ij}$ for distinct $i,j,k,l$;
		\item $x_{ij}x_{jk}+x_{jk}x_{ki}+x_{ki}x_{ij}=0$ for distinct $i,j,k$.
	\end{itemize}
	
	For any graph $G$ with vertex set $[n]$, the \emph{Fomin-Kirillov algebra} ${\FK_G}$ of $G$ is the subalgebra of $\FK_n$ generated by $x_{ij}$ for all edges $\overline{ij}$ in $G$.
\end{defn}

In particular, $\FK_n = \FK_{K_n}$, where $K_n$ is the complete graph with vertex set $[n]$. Note that since the set of relations of $\FK_n$ is fixed by relabelings of the vertex set, the structure of $\FK_G$ depends only on the structure of the graph $G$ up to isomorphism.

Typically $\FK_G$ will have minimal relations that are not quadratic. The most important relations for our purposes will be the following, proved in \cite{BLM} by a straightforward induction.

\begin{prop}[\cite{BLM}]\label{prop-cycrel}
For distinct $i_1, i_2, \dots, i_m \in [n]$,
\[x_{i_1i_2}x_{i_2i_3}\cdots x_{i_{m-1}i_m} + x_{i_2i_3}x_{i_3i_4}\cdots x_{i_{m-1}i_m}x_{i_mi_1} + \cdots + x_{i_mi_1}x_{i_1i_2}\cdots x_{i_{m-2}i_{m-1}} = 0.\]
\end{prop}
When $m=3$, this is the third quadratic relation in the definition of $\FK_n$.

\subsection{Grading}

The algebras $\FK_G$ have three natural gradings:
\begin{itemize}
\item Any monomial $P \in \FK_G$ has the usual notion of degree, which we denote $d(P)$.
\item There is a grading with respect to the symmetric group $\S_n$: we define the $\S_n$-degree of $x_{ij}$ to be the transposition $(i\;j) \in \S_n$ and extend to all monomials by multiplicativity. We denote the $\S_n$-degree of $P$ by $\sigma_P$.
\item For any monomial $P \in \FK_G$, let $\supp(P)$ be the subgraph of $G$ with edges $\overline{ij}$ for $x_{ij}$ appearing in $P$. Then we can define a grading on $\FK_G$ by letting $\Pi(P)$ be the set partition of $[n]$ that gives the connected components of $\supp(P)$. For example, $\Pi(x_{12}x_{23}x_{45}x_{31}) = 123|45$.
\end{itemize}
It is easy to check that the relations of $\FK_n$ are homogeneous with respect to all three of these gradings. Note that if $P, Q \in \FK_G$ are homogeneous with respect to all three of these gradings, then $d(PQ) = d(P)+d(Q)$, $\sigma_{PQ} = \sigma_P\sigma_Q$, and $\Pi(PQ)$ is the common coarsening of $\Pi(P)$ and $\Pi(Q)$.

If $G$ and $H$ are graphs with disjoint vertex sets, then by considering $\Pi$-degree, we see that the only relations of $\FK_{G+H}$ are either commuting relations between $\FK_G$ and $\FK_H$, or else they follow from relations within $\FK_G$ and $\FK_H$. Therefore $\FK_{G+H} \cong \FK_G \otimes \FK_H$.

\begin{defn}
	A monomial $P \in \FK_G$ is \emph{simple} if $\sigma_P$ has exactly $n-d(P)$ cycles. We denote by $\FK_G^s$ the quotient of $\FK_G$ by all non-simple monomials in $\FK_G$, and we denote by $\FK_G^\sigma$ the span of all simple monomials of $\S_n$-degree $\sigma$.
\end{defn}
There is a straightforward characterization of simple monomials.
\begin{prop} \label{prop-simple}
A monomial $P \in \FK_G$ is simple if and only if $P$ has no repeated variables and $\supp(P)$ has no cycles. If $P$ is simple, then $\Pi(P)$ is the set partition given by the cycles of $\sigma(P)$. (Hence $\FK_G^\sigma$ is $\Pi$-homogeneous.)
\end{prop}
\begin{proof}
Multiplying a permutation by a transposition $(i\; j)$ either merges the two distinct cycles containing $i$ and $j$ or splits apart the cycle containing both $i$ and $j$. Thus if $P$ is simple, then multiplying by each successive transposition in $\sigma_P = (i_1\;j_1)\cdots (i_d\;j_d)$ must merge two cycles. Hence $\Pi(P)$ is given by the cycles of $\sigma_P$. Since $P$ has the minimum possible degree given $\sigma_P$, $\supp(P)$ must have the minimum possible number of edges given $\Pi(P)$, so there can be no repeated edges or cycles in $\supp(P)$.
\end{proof}

If $\sigma$ has cycle decomposition $\sigma^{(1)}\sigma^{(2)} \dots \sigma^{(k)}$, then multiplication gives an isomorphism $\bigotimes_i \FK_{G_i}^{\sigma^{(i)}} \cong \FK_G^\sigma$, where $G_i$ is the induced subgraph of $G$ on the vertices in $\sigma^{(i)}$. 

\subsection{Orlik-Terao algebra}

In \cite{OT}, Orlik and Terao define a commutative algebra $U(\mathscr A)$ for any hyperplane arrangement $\mathscr A$. When $\mathscr A$ is specialized to the graphical arrangement corresponding to a graph $G$, we get the following definition.

\begin{defn}
	Let $G$ be a graph on $n$ vertices. The \emph{(Artinian) Orlik-Terao algebra} $U_G$ is the commutative ($\QQ$-)algebra with generators $u_{ij} = -u_{ji}$ for all edges $\overline{ij}$ in $G$ subject to the following relations:
	\begin{itemize}
		\item a monomial $P \in U_G$ vanishes if $P$ has a repeated variable or $\supp(P)$ has a cycle; and 
		\item for any linear dependence
		\[\sum_{p=1}^k c_p \cdot (e_{i_p}-e_{j_p}) = 0 \in \RR^n,\]
		we have the relation 
		\[\sum_{p=1}^k \left(c_p \prod_{q\neq p} u_{i_qj_q}\right) = 0.\]
	\end{itemize}
\end{defn}

Below is a summary of the results of \cite{OT} that we will need. Fix a total order on the edges of $G$, and define a \emph{broken circuit} to be a cycle with its minimum edge removed. 

\begin{thm} [\cite{OT}] \label{thm-ot}
	\begin{enumerate}[(a)]
		\item The relations in $U_G$ are generated by $u_{ij}^2=0$ and
		\[u_{i_1i_2}u_{i_2i_3}\cdots u_{i_{m-1}i_m} + u_{i_2i_3}u_{i_3i_4}\cdots u_{i_{m-1}i_m}u_{i_mi_1} + \cdots + u_{i_mi_1}u_{i_1i_2}\cdots u_{i_{m-2}i_{m-1}} = 0,\]
		where $i_1, i_2, \dots, i_m$ is a cycle in $G$.
		\item The monomials $\prod_{\overline{ij}\in S}u_{ij}$, where $S$ is any subset of edges of $G$ that does not contain a broken circuit, form a basis of $U_G$.
		\item The Hilbert series of $U_G$ is $(-t)^n \chi_G(-t^{-1})$, where $\chi_G$ is the chromatic polynomial of $G$.
	\end{enumerate}
\end{thm}

We call the basis in part (b) a \emph{no-broken-circuit (NBC) basis}. The Orlik-Terao algebra is known to be isomorphic as graded vector space to related algebras such as the Orlik-Solomon algebra or the Gelfand-Varchenko algebra \cite{OrlikSolomon, GelfandVarchenko}.

We will also need the following simple consequence of Theorem~\ref{thm-ot}.
\begin{prop} \label{prop-otsub}
	For any subgraph $H \subset G$, $U_H$ is a subalgebra of $U_G$.
\end{prop}
\begin{proof}
	Any relation in $U_H$ is also a relation in $U_G$, so $U_H$ maps surjectively onto the subalgebra of $U_G$ generated by $u_{ij}$ for $\overline{ij}\in H$. The NBC-basis for $U_H$ is a subset of the NBC-basis for $U_G$, so this map is injective as well.
\end{proof}

\subsection{Commutative quotient}

Let $\FK_G^{ab}$ be the commutative quotient of $\FK_G$. In other words, let $\pi\colon \FK_G \to \FK_G^{ab}$ be the quotient map by the ideal generated by $x_{ij}x_{ik}-x_{ik}x_{ij}$ for $\overline{ij}, \overline{ik} \in G$. The main result of this section is the following theorem.

\begin{thm}\label{thm-main}
	The commutative quotient $\FK_G^{ab}$ and the Orlik-Terao algebra $U_G$ are isomorphic.
\end{thm}
\begin{proof}
	By the defining relations of $\FK_n$, Proposition~\ref{prop-cycrel}, and Theorem~\ref{thm-ot}(a), $\FK_n^{ab} \cong U_{K_n}$. Hence by Proposition~\ref{prop-otsub}, the subalgebra of $\FK_n^{ab}$ generated by $x_{ij}$ for $\overline{ij} \in G$, which is a quotient of $\FK_G^{ab}$, is isomorphic to $U_G$. But $U_G$ is generated by the relations given in Theorem~\ref{thm-ot}(a), and all of these relations also occur in $\FK_G^{ab}$ by Proposition~\ref{prop-cycrel}. Thus $U_G \cong \FK_G^{ab}$.
\end{proof}

Using Theorem~\ref{thm-ot}(c), we immediately get the following corollary conjectured by Kirillov in \cite{Kirillov}.

\begin{cor}\label{cor-main}
The Hilbert series of $\FK_G^{ab}$ is 
\[\Hilb_G^{ab} (t) = (-t)^n \chi_G(-t^{-1}),\]
where $\chi_G(t)$ is the chromatic polynomial of $G$.
\end{cor}

Note that Theorem~\ref{thm-main} together with Proposition~\ref{prop-simple} implies that for a monomial $P \in \FK_G$, $\pi(P)=0$ unless $P$ is simple. Thus $\FK_G^s$ is, in a sense, a noncommutative analogue of $U_G$. We will study the structure of each $\FK_G^\sigma$ in the next section.

\section{Noncrossing trees}

In this section, we will describe the relations in $\FK_G^\sigma$ using the combinatorics of noncrossing forests. Recall from above that if $\sigma$ has cycle decomposition $\sigma_1\sigma_2 \dots \sigma_k$, then $\FK_G^\sigma \cong \bigotimes_i \FK_{G_i}^{\sigma_i}$, where $G_i$ is the induced subgraph of $G$ on the vertices in $\sigma_i$. Hence it suffices to consider the case when $\sigma$ is an $n$-cycle. For the rest of this section, we will assume without loss of generality that $\sigma = (1\;2\;\cdots \; n)$. We may also assume that all monomials are written in terms of $x_{ij}$ with $\overline{ij} \in G$ and $i<j$.

\subsection{Noncrossing} We begin by defining noncrossing graphs.

\begin{defn}
A graph with vertex set $[n]$ is \emph{noncrossing} if it does not contain edges $\overline{ac}$ and $\overline{bd}$ with $a<b<c<d$.
\end{defn}

We can draw a noncrossing graph as follows: draw the vertices $1$ through $n$ from left to right on a horizontal line. Then the edges of the graph can be drawn as arcs above this line in such a way that no two arcs cross. (See Figure~\ref{fig-noncrossing}.) Note that the definition of noncrossing depends only on the cyclic ordering of $[n]$.

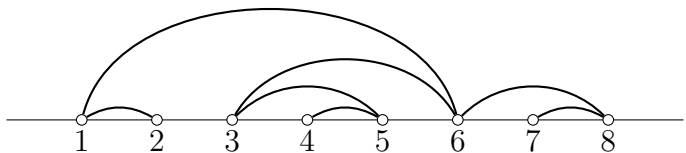
\begin{figure}
	\begin{tikzpicture}
		\draw[thin] (0,0) to (9,0);
		\node[v] (1) at (1,0){};
		\node[v] (2) at (2,0){};
		\node[v] (3) at (3,0){};
		\node[v] (4) at (4,0){};
		\node[v] (5) at (5,0){};
		\node[v] (6) at (6,0){};
		\node[v] (7) at (7,0){};
		\node[v] (8) at (8,0){};
		\draw[thick] (1) to [bend left=30] (2);
		\draw[thick] (1) to [bend left=75] (6);
		\draw[thick] (3) to [bend left=60] (6);
		\draw[thick] (3) to [bend left=45] (5);
		\draw[thick] (4) to [bend left=30] (5);
		\draw[thick] (6) to [bend left=45] (8);
		\draw[thick] (7) to [bend left=30] (8);
		\draw (1) node[below]{1};
		\draw (2) node[below]{2};
		\draw (3) node[below]{3};
		\draw (4) node[below]{4};
		\draw (5) node[below]{5};
		\draw (6) node[below]{6};
		\draw (7) node[below]{7};
		\draw (8) node[below]{8};
	\end{tikzpicture}
	\caption{\label{fig-noncrossing} A noncrossing tree $T$ on $[8]$. See Example~\ref{ex-noncrossing}.}
\end{figure}

\begin{defn}
Let $P$ be a simple monomial, and let $T$ be a noncrossing tree containing $\supp(P)$. We say that $P$ \emph{respects} $T$ if the left-to-right order of the variables containing a fixed index $i$ in $P$ is the same as the clockwise order of the edges incident to $i$ in $T$.
\end{defn}

\begin{ex} \label{ex-noncrossing}
	Let $T$ be the noncrossing tree shown in Figure~\ref{fig-noncrossing}. There are three edges incident to vertex 6, namely $\overline{36}$, $\overline{16}$, and $\overline{68}$ in clockwise order. Hence if $P$ is a monomial with $\supp(P)=T$ that respects $T$, then it must contain $x_{36}$, $x_{16}$, and $x_{68}$ in that order. Similarly, $x_{16}$ must appear before $x_{12}$; $x_{36}$ before $x_{35}$; $x_{45}$ before $x_{35}$; and $x_{78}$ before $x_{68}$. One such monomial is $x_{36}x_{45}x_{35}x_{16}x_{12}x_{78}x_{68}$. Another is $x_{45}x_{78}x_{36}x_{16}x_{35}x_{68}x_{12}$. Observe that these two monomials can be obtained from one another using only commutation relations.
\end{ex}

We first show that we can use noncrossing trees to represent simple monomials.
\begin{prop} \label{prop-tree}
The map that sends a monomial $P\in\FK_G^\sigma$ to $\supp(P)$ is a bijection between monomials in $\FK_G^\sigma$ up to commuting relations and noncrossing trees with vertex set $[n]$. Moreover, for any noncrossing tree $T$, the simple monomials $P$ with $\supp(P)=T$ that respect $T$ are exactly those that lie in $\FK_G^\sigma$.
\end{prop}
\begin{proof}
Let $T$ be a noncrossing tree. Call an edge $\overline{ij}$ in $T$ \emph{terminal} if it is the furthest clockwise edge at both $i$ and $j$. Clearly no two terminal edges can share an endpoint. Moreover, every noncrossing tree has a terminal edge: if none of the rightmost edges at each vertex were the same, then there would be a cycle among these $n$ distinct edges. It follows $T$ is respected by some monomial $P$: construct $P$ from right-to-left by removing terminal edges from $T$.

We first show by induction that given a simple monomial $P$ with $\supp(P)=T$, $\sigma_P=\sigma$ if and only if $P$ respects $T$. We may write $P=Qx_{ij}$ with $Q$ simple. If $\sigma_P=\sigma$, then 
\[\sigma_Q = \sigma \cdot (i\;\;j) = (i+1\;\;i+2\;\;\cdots\;\;j)(j+1\;\;j+2\;\;\cdots\;\;n\;\;1\;\;\cdots\;\;i),\]
so we may write $Q=Q_1Q_2$, where $Q_1$ contains those variables of $Q$ with indices $i+1, \dots, j$ (in the same order that they appear in $Q$) and $Q_2$ contains the remaining variables. Here $Q_1$ and $Q_2$ are simple, and $\sigma_{Q_1}$ and $\sigma_{Q_2}$ are the two factors above. By induction, it follows that $Q_1$ and $Q_2$ respect $T$, so $Q$ does also. Since $x_{ij}$ appears last in $P$ and $\overline{ij}$ is terminal in $T$, $P$ also respects $T$. The reverse direction is similar.

To complete the proof, we will show that any two simple monomials with the same support that both respect $T$ are equivalent up to commuting, which we denote by $\sim$. We proceed by induction on degree. Suppose $P=Qx_{ij}$ and $P'=Q'x_{i'j'}$ respect $T$. If $x_{ij}=x_{i'j'}$, then by induction $Q \sim Q'$, so $P \sim P'$. If not, then both $x_{ij}$ and $x_{i'j'}$ are terminal edges, so they are disjoint. Since $x_{i'j'}$ is also a terminal edge in $T \backslash \overline{ij}$, by induction, $Q \sim Rx_{i'j'}$ for some $R$, so $P=Qx_{ij} \sim Rx_{i'j'}x_{ij} \sim Rx_{ij}x_{i'j'} \sim Q'x_{i'j'}= P'$.
\end{proof}

\subsection{Reduction}

In light of Proposition~\ref{prop-tree}, we will use $x_T$ to denote the element of $\FK_G^\sigma$ corresponding to a noncrossing tree $T$. (We extend this to noncrossing forests in the obvious way.) Translating Proposition~\ref{prop-cycrel} into this language, we get the following proposition. See Figure~\ref{fig-cycrel}.

\begin{figure}
\begin{align*}
\vc{
\begin{tikzpicture}[scale=0.7]
\draw[thin] (0,0) to (9,0);
\node[v] (1) at (1,0){};
\node[v] (2) at (2,0){};
\node[v] (3) at (3,0){};
\node[v] (4) at (4,0){};
\node[v] (5) at (5,0){};
\node[v] (6) at (6,0){};
\node[v] (7) at (7,0){};
\node[v] (8) at (8,0){};
\draw[thick] (1) to [bend left=30] (2);
\draw[ultra thick, red] (1) to [bend left=60] (8);
\draw[ultra thick, red] (4) to [bend left=45] (6);
\draw[thick] (3) to [bend left=30] (4);
\draw[thick] (4) to [bend left=30] (5);
\draw[ultra thick, red] (1) to [bend left=45] (4);
\draw[thick] (7) to [bend left=30] (8);
\end{tikzpicture}
}\phantom{{}-{}}
= & \phantom{{}-{}}
\vc{
	\begin{tikzpicture}[scale=0.7]
	\draw[thin] (0,0) to (9,0);
	\node[v] (1) at (1,0){};
	\node[v] (2) at (2,0){};
	\node[v] (3) at (3,0){};
	\node[v] (4) at (4,0){};
	\node[v] (5) at (5,0){};
	\node[v] (6) at (6,0){};
	\node[v] (7) at (7,0){};
	\node[v] (8) at (8,0){};
	\draw[ultra thick, white] (1) to [bend left=60] (8);
	\draw[thick] (1) to [bend left=30] (2);
	\draw[ultra thick, red] (6) to [bend left=60] (8);
	\draw[ultra thick, red] (4) to [bend left=45] (6);
	\draw[thick] (3) to [bend left=30] (4);
	\draw[thick] (4) to [bend left=30] (5);
	\draw[ultra thick, red] (1) to [bend left=45] (4);
	\draw[thick] (7) to [bend left=30] (8);
	\end{tikzpicture}
}\\
& {}-
\vc{
	\begin{tikzpicture}[scale=0.7]
	\draw[thin] (0,0) to (9,0);
	\node[v] (1) at (1,0){};
	\node[v] (2) at (2,0){};
	\node[v] (3) at (3,0){};
	\node[v] (4) at (4,0){};
	\node[v] (5) at (5,0){};
	\node[v] (6) at (6,0){};
	\node[v] (7) at (7,0){};
	\node[v] (8) at (8,0){};
	\draw[thick] (1) to [bend left=30] (2);
	\draw[ultra thick, red] (1) to [bend left=60] (8);
	\draw[ultra thick, red] (4) to [bend left=45] (6);
	\draw[thick] (3) to [bend left=30] (4);
	\draw[thick] (4) to [bend left=30] (5);
	\draw[ultra thick, red] (6) to [bend left=45] (8);
	\draw[thick] (7) to [bend left=30] (8);
	\end{tikzpicture}
}\\
& {}-
\vc{
	\begin{tikzpicture}[scale=0.7]
	\draw[thin] (0,0) to (9,0);
	\node[v] (1) at (1,0){};
	\node[v] (2) at (2,0){};
	\node[v] (3) at (3,0){};
	\node[v] (4) at (4,0){};
	\node[v] (5) at (5,0){};
	\node[v] (6) at (6,0){};
	\node[v] (7) at (7,0){};
	\node[v] (8) at (8,0){};
	\draw[thick] (1) to [bend left=30] (2);
	\draw[ultra thick, red] (1) to [bend left=60] (8);
	\draw[ultra thick, red] (6) to [bend left=45] (8);
	\draw[thick] (3) to [bend left=30] (4);
	\draw[thick] (4) to [bend left=30] (5);
	\draw[ultra thick, red] (1) to [bend left=45] (4);
	\draw[thick] (7) to [bend left=30] (8);
	\end{tikzpicture}
}
\end{align*}
\caption{\label{fig-cycrel}
	A reduction of the form \eqref{eq-rel} for $m=4$. See Proposition~\ref{prop-reduce}.
	}
\end{figure}
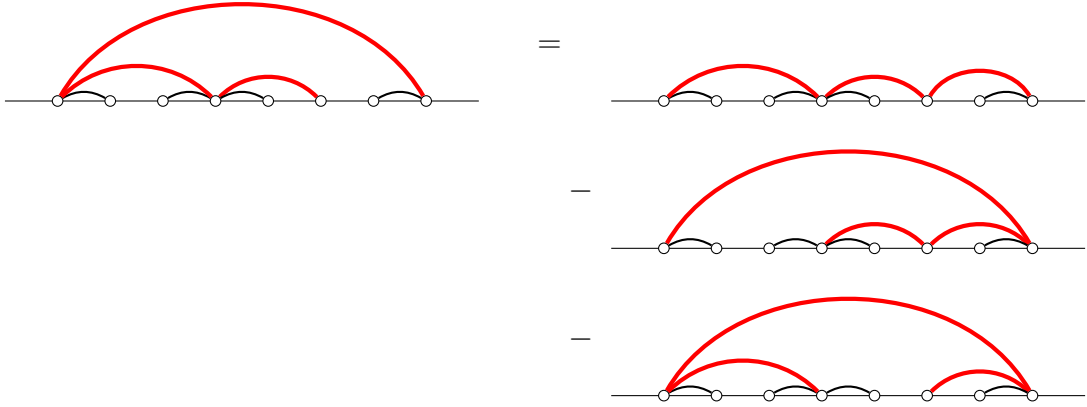

\begin{prop}\label{prop-reduce}
Let $1 \leq i_1<i_2<\dots<i_m \leq n$, and let $H$ be a connected noncrossing graph containing the unique cycle $C=(i_1, i_2, \dots, i_m)$. Let $T_1 = H \backslash \overline{i_1i_m}$ and $T_j = H \backslash \overline{i_{j-1}i_{j}}$ for $1 < j \leq m$. Then
\begin{equation*}\label{eq-rel}
x_{T_m} = x_{T_1}-x_{T_2}-x_{T_3}-\cdots - x_{T_{m-1}}. \tag{$*$} 
\end{equation*}
\end{prop}
\begin{proof}
We claim that $x_{T_j} = x_L\cdot x_{j,j+1}\cdots x_{m-1, m}x_{1m}x_{12}\cdots x_{j-2,j-1}\cdot x_R$ for some subgraphs $L$ and $R$ such that $L$, $R$, and $C$ partition the edges of $H$. Indeed, $L$ is the forest consisting of those edges attached to $C$ by an edge to the left of (that is, counterclockwise from) an edge in $C$, and likewise $R$ is similarly defined to the right. The result now follows easily from Proposition~\ref{prop-cycrel}.
\end{proof}

We can use the equality in Proposition~\ref{prop-reduce} to reduce any element of $\FK_G^\sigma$ by replacing the left side with the right side. The monomials that will result will be precisely those that cannot be reduced any further.

\begin{defn}
A noncrossing tree $T$ with vertex set $[n]$ is \emph{$G$-reduced} if $T \subset G$ and there do not exist $1 \leq i_1 < i_2 < \dots < i_m \leq n$ such that $\overline{i_1i_m} \in T$;  $\overline{i_{j-1}i_j} \in T$ for all $1 < j < m$; $\overline{i_{m-1}i_m} \in G$; and $T \cup \overline{i_{m-1}i_m}$ is noncrossing.
\end{defn}

In other words, a noncrossing tree $T$ is $G$-reduced if it cannot appear on the left side of a relation of the form \eqref{eq-rel} in $\FK_G^\sigma$.
\begin{lemma} \label{lemma-spanning}
The set $\{x_T\}$, where $T$ ranges over all $G$-reduced noncrossing trees, spans $\FK_G^\sigma$.
\end{lemma}
\begin{proof}
It suffices to show that starting from an element of $\FK_G^\sigma$ and repeatedly replacing the left side of \eqref{eq-rel} with right side, this process will eventually terminate. Order the edges $\overline{ij}$ of $G$ (with $i<j$) lexicographically. Any tree appearing on the right side of \eqref{eq-rel} is obtained from the tree on the left side by replacing an edge in the cycle $C=(i_1, i_2, \dots, i_m)$ with the edge $\overline{i_{m-1}i_m}$. But this edge is the largest edge of $C$, so applying \eqref{eq-rel} strictly increases the set of edges lexicographically. Therefore the reduction must terminate.
\end{proof}
In fact, we will show that the $G$-reduced noncrossing trees give a basis for $\FK_G^\sigma$. This is equivalent to showing that the order in which we perform reductions of the form \eqref{eq-rel} does not matter, so that one will always at the same answer regardless of choices made during the reduction.

To prove this, we first show it in the case when $G=K_n$. In this case, the $K_n$-reduced noncrossing trees are those for which, at any vertex $i$, there is at most one edge $\overline{ij}$ with $j>i$. The number of such trees is the Catalan number $C_{n-1} = \frac1n\binom{2n-2}{n-1}$.

\begin{rmk}
	One can also prove Lemma~\ref{lemma-knbasis} below using Theorem~\ref{thm-main} by showing that the images of $K_n$-reduced noncrossing trees in $\FK_n^{ab} \cong U_{K_n}$ lie in an NBC-basis under an appropriate ordering on the edges. However, such an argument cannot be applied directly for most $G$ as the missing edges cause $G$-reduced trees and broken circuits not to have as simple a description in general.
\end{rmk}

\begin{rmk}
	In \cite{Meszaros}, a similar reduction on noncrossing trees was given based on subdivisions of root polytopes. While that reduction algorithm can also be used to give a basis of $\FK_n^{s}$, albeit with a different term order, it usually fails to give unique reductions for $\FK_G^s$ when $G$ is not the complete graph.
\end{rmk}

\begin{lemma} \label{lemma-knbasis}
The set $\{x_T\}$, where $T$ ranges over all $K_n$-reduced noncrossing trees, is a basis of $\FK_n^\sigma$.
\end{lemma}
\begin{proof}
Note that by the definition of $\FK_n$ (and the fact that simple monomials have no repeated variables), all of the relations between $x_T$ in $\FK_n^\sigma$ can be derived from the three-term quadratic relation $x_{ik}x_{ij} = x_{ij}x_{jk} - x_{jk}x_{ik}$ for $i<j<k$; in other words, they are given by relations of the form \eqref{eq-rel} when $m=3$.

Consider the reduction procedure that reduces an element of $\FK_n$ (written as a linear combination of $x_T$) by repeatedly replacing the left side of \eqref{eq-rel} for $m=3$ with the right side. We claim that the result is independent of the specific sequence of reductions. This will follow from the Diamond Lemma if we can show that the reduction is locally confluent, that is, if an element $a$ can be reduced to either $b$ or $c$ by one application of \eqref{eq-rel}, then it is possible to reduce both $b$ and $c$ to the same expression $d$ (possibly with multiple applications of \eqref{eq-rel}).

Suppose $x_T$ can be reduced in two different ways, so that it contains edges $\overline{ik}$, $\overline{ij}$, $\overline{i'k'}$, and $\overline{i'j'}$ with $i<j<k$ and $i'<j'<k'$. If all of these edges are distinct, then the two reductions can be applied in either order with the same result.

If instead the edges are not disjoint, let us assume without loss of generality that $i=i'=1$, $j'=2$, $j=k'=3$, and $k=4$. Then we have the following two reductions of $x_{14}x_{13}x_{12}$:
\begin{align}
x_{14}x_{13}x_{12} &= x_{13}x_{34}x_{12} - x_{34}x_{14}x_{12},\\
x_{14}x_{13}x_{12} &= x_{14}x_{12}x_{23} - x_{14}x_{23}x_{13}.
\end{align}
We can further reduce (1) as follows:
\begin{align*}
x_{13}x_{12}x_{34}-x_{34}x_{14}x_{12} &= (x_{12}x_{23}- x_{23}x_{13})x_{34}-x_{34}(x_{12}x_{24}-x_{24}x_{14})\\
&= x_{12}x_{23}x_{34} - x_{23}x_{13}x_{34} - x_{34}x_{12}x_{24} + x_{34}x_{24}x_{14}.
\end{align*}
We can further reduce (2) to the same expression as follows:
\begin{align*}
x_{14}x_{12}x_{23} - x_{23}x_{14}x_{13} &= (x_{12}x_{24}- x_{24}x_{14})x_{23} - x_{23}(x_{13}x_{34} - x_{34}x_{14})\\
&= x_{12}x_{24}x_{23} -x_{24}x_{23}x_{14}- x_{23}x_{13}x_{34}+x_{23}x_{34}x_{14}\\
&= x_{12}(x_{23}x_{34}-x_{34}x_{24})-(x_{23}x_{34}-x_{34}x_{24})x_{14}-x_{23}x_{13}x_{34}+x_{23}x_{34}x_{14}\\
&= x_{12}x_{23}x_{34}-x_{12}x_{34}x_{24}+x_{34}x_{24}x_{14} - x_{23}x_{13}x_{34}.\end{align*}
See Figure~\ref{fig-proof} for a graphical depiction of this calculation.

It follows that the reduction procedure always yields a unique result. Thus the set $\{x_T\}$, where $T$ ranges over all $K_n$-reduced noncrossing trees, is linearly independent, so by Lemma~\ref{lemma-spanning} it is a basis of $\FK_n^\sigma$.
\end{proof}

\begin{figure}
	\begin{align*}
	\vc{
		\begin{tikzpicture}[scale=0.7]
		\draw[thin] (0,0) to (5,0);
		\node[v] (1) at (1,0){};
		\node[v] (2) at (2,0){};
		\node[v] (3) at (3,0){};
		\node[v] (4) at (4,0){};
		\draw[very thick, red] (1) to [bend left=60] (4);
		\draw[very thick, red] (1) to [bend left=45] (3);
		\draw[thick] (1) to [bend left=30] (2);
		\end{tikzpicture}
	}
	\phantom{{}-{}}=&\phantom{{}-{}}
	\vc{
		\begin{tikzpicture}[scale=0.7]
		\draw[thin, red] (0,0) to (5,0);
		\node[v] (1) at (1,0){};
		\node[v] (2) at (2,0){};
		\node[v] (3) at (3,0){};
		\node[v] (4) at (4,0){};
		\draw[thick, white] (1) to [bend left=60] (4);
		\draw[thick] (3) to [bend left=30] (4);
		\draw[very thick, blue] (1) to [bend left=45] (3);
		\draw[very thick, blue] (1) to [bend left=30] (2);
		\end{tikzpicture}
	}
	-
	\vc{
		\begin{tikzpicture}[scale=0.7]
		\draw[thin, red] (0,0) to (5,0);
		\node[v] (1) at (1,0){};
		\node[v] (2) at (2,0){};
		\node[v] (3) at (3,0){};
		\node[v] (4) at (4,0){};
		\draw[thick] (3) to [bend left=30] (4);
		\draw[very thick, green] (1) to [bend left=60] (4);
		\draw[very thick, green] (1) to [bend left=30] (2);
		\end{tikzpicture}
	}\\
	=&\phantom{{}-{}}
	\vc{
		\begin{tikzpicture}[scale=0.7]
		\draw[thin, blue] (0,0) to (5,0);
		\node[v] (1) at (1,0){};
		\node[v] (2) at (2,0){};
		\node[v] (3) at (3,0){};
		\node[v] (4) at (4,0){};
		\draw[thick, white] (1) to [bend left=60] (4);
		\draw[thick] (3) to [bend left=30] (4);
		\draw[thick] (2) to [bend left=30] (3);
		\draw[thick] (1) to [bend left=30] (2);
		\end{tikzpicture}
	}
	-
	\vc{
		\begin{tikzpicture}[scale=0.7]
		\draw[thin, blue] (0,0) to (5,0);
		\node[v] (1) at (1,0){};
		\node[v] (2) at (2,0){};
		\node[v] (3) at (3,0){};
		\node[v] (4) at (4,0){};
		\draw[thick, white] (1) to [bend left=60] (4);
		\draw[thick] (3) to [bend left=30] (4);
		\draw[thick] (1) to [bend left=60] (3);
		\draw[thick] (2) to [bend left=30] (3);
		\end{tikzpicture}
	}\\
	&{}-{}
	\vc{
		\begin{tikzpicture}[scale=0.7]
		\draw[thin, green] (0,0) to (5,0);
		\node[v] (1) at (1,0){};
		\node[v] (2) at (2,0){};
		\node[v] (3) at (3,0){};
		\node[v] (4) at (4,0){};
		\draw[thick, white] (1) to [bend left=60] (4);
		\draw[thick] (3) to [bend left=30] (4);
		\draw[thick] (2) to [bend left=60] (4);
		\draw[thick] (1) to [bend left=30] (2);
		\end{tikzpicture}
	}
	+
	\vc{
		\begin{tikzpicture}[scale=0.7]
		\draw[thin, green] (0,0) to (5,0);
		\node[v] (1) at (1,0){};
		\node[v] (2) at (2,0){};
		\node[v] (3) at (3,0){};
		\node[v] (4) at (4,0){};
		\draw[thick] (3) to [bend left=30] (4);
		\draw[thick] (1) to [bend left=60] (4);
		\draw[thick] (2) to [bend left=45] (4);
		\end{tikzpicture}
	}\\
	\vc{
		\begin{tikzpicture}[scale=0.7]
		\draw[thin] (0,0) to (5,0);
		\node[v] (1) at (1,0){};
		\node[v] (2) at (2,0){};
		\node[v] (3) at (3,0){};
		\node[v] (4) at (4,0){};
		\draw[thick] (1) to [bend left=60] (4);
		\draw[very thick, red] (1) to [bend left=45] (3);
		\draw[very thick, red] (1) to [bend left=30] (2);
		\end{tikzpicture}
	}
	\phantom{{}-{}}=&\phantom{{}-{}}
	\vc{
		\begin{tikzpicture}[scale=0.7]
		\draw[thin, red] (0,0) to (5,0);
		\node[v] (1) at (1,0){};
		\node[v] (2) at (2,0){};
		\node[v] (3) at (3,0){};
		\node[v] (4) at (4,0){};
		\draw[very thick, blue] (1) to [bend left=60] (4);
		\draw[thick] (2) to [bend left=30] (3);
		\draw[very thick, blue] (1) to [bend left=30] (2);
		\end{tikzpicture}
	}
	-
	\vc{
		\begin{tikzpicture}[scale=0.7]
		\draw[thin, red] (0,0) to (5,0);
		\node[v] (1) at (1,0){};
		\node[v] (2) at (2,0){};
		\node[v] (3) at (3,0){};
		\node[v] (4) at (4,0){};
		\draw[thick] (2) to [bend left=30] (3);
		\draw[very thick, green] (1) to [bend left=60] (4);
		\draw[very thick, green] (1) to [bend left=45] (3);
		\end{tikzpicture}
	}\\
	=&\phantom{{}-{}}
	\vc{
		\begin{tikzpicture}[scale=0.7]
		\draw[thin, blue] (0,0) to (5,0);
		\node[v] (1) at (1,0){};
		\node[v] (2) at (2,0){};
		\node[v] (3) at (3,0){};
		\node[v] (4) at (4,0){};
		\draw[thick, white] (1) to [bend left=60] (4);
		\draw[very thick, violet] (2) to [bend left=30] (3);
		\draw[very thick, violet] (2) to [bend left=45] (4);
		\draw[thick] (1) to [bend left=30] (2);
		\end{tikzpicture}
	}
	-
	\vc{
		\begin{tikzpicture}[scale=0.7]
		\draw[thin, blue] (0,0) to (5,0);
		\node[v] (1) at (1,0){};
		\node[v] (2) at (2,0){};
		\node[v] (3) at (3,0){};
		\node[v] (4) at (4,0){};
		\draw[very thick, orange] (2) to [bend left=45] (4);
		\draw[thick] (1) to [bend left=60] (4);
		\draw[very thick, orange] (2) to [bend left=30] (3);
		\end{tikzpicture}
	}\\
	&{}-
	\vc{
		\begin{tikzpicture}[scale=0.7]
		\draw[thin, green] (0,0) to (5,0);
		\node[v] (1) at (1,0){};
		\node[v] (2) at (2,0){};
		\node[v] (3) at (3,0){};
		\node[v] (4) at (4,0){};
		\draw[thick, white] (1) to [bend left=60] (4);
		\draw[thick] (3) to [bend left=30] (4);
		\draw[thick] (2) to [bend left=30] (3);
		\draw[thick] (1) to [bend left=45] (3);
		\end{tikzpicture}
	}
	+
	\vc{
		\begin{tikzpicture}[scale=0.7]
		\draw[thin, green] (0,0) to (5,0);
		\node[v] (1) at (1,0){};
		\node[v] (2) at (2,0){};
		\node[v] (3) at (3,0){};
		\node[v] (4) at (4,0){};
		\draw[thick] (1) to [bend left=60] (4);
		\draw[thick] (3) to [bend left=30] (4);
		\draw[thick] (2) to [bend left=30] (3);
		\end{tikzpicture}
	}\\
	=&\phantom{{}-{}}
	\vc{
		\begin{tikzpicture}[scale=0.7]
		\draw[thin, violet] (0,0) to (5,0);
		\node[v] (1) at (1,0){};
		\node[v] (2) at (2,0){};
		\node[v] (3) at (3,0){};
		\node[v] (4) at (4,0){};
		\draw[thick, white] (1) to [bend left=60] (4);
		\draw[thick] (3) to [bend left=30] (4);
		\draw[thick] (2) to [bend left=30] (3);
		\draw[thick] (1) to [bend left=30] (2);
		\end{tikzpicture}
	}
	-	
	\vc{
		\begin{tikzpicture}[scale=0.7]
		\draw[thin, violet] (0,0) to (5,0);
		\node[v] (1) at (1,0){};
		\node[v] (2) at (2,0){};
		\node[v] (3) at (3,0){};
		\node[v] (4) at (4,0){};
		\draw[thick, white] (1) to [bend left=60] (4);
		\draw[thick] (3) to [bend left=30] (4);
		\draw[thick] (2) to [bend left=60] (4);
		\draw[thick] (1) to [bend left=30] (2);
		\end{tikzpicture}
	}
	\\
	&{}-{}
	\vc{
		\begin{tikzpicture}[scale=0.7]
		\draw[thin, orange] (0,0) to (5,0);
		\node[v] (1) at (1,0){};
		\node[v] (2) at (2,0){};
		\node[v] (3) at (3,0){};
		\node[v] (4) at (4,0){};
		\draw[thick] (1) to [bend left=60] (4);
		\draw[thick] (3) to [bend left=30] (4);
		\draw[thick] (2) to [bend left=30] (3);
		\end{tikzpicture}
	}
	+
	\vc{
		\begin{tikzpicture}[scale=0.7]
		\draw[thin, orange] (0,0) to (5,0);
		\node[v] (1) at (1,0){};
		\node[v] (2) at (2,0){};
		\node[v] (3) at (3,0){};
		\node[v] (4) at (4,0){};
		\draw[thick] (3) to [bend left=30] (4);
		\draw[thick] (1) to [bend left=60] (4);
		\draw[thick] (2) to [bend left=45] (4);
		\end{tikzpicture}
	}\\
	&{}-
	\vc{
		\begin{tikzpicture}[scale=0.7]
		\draw[thin, green] (0,0) to (5,0);
		\node[v] (1) at (1,0){};
		\node[v] (2) at (2,0){};
		\node[v] (3) at (3,0){};
		\node[v] (4) at (4,0){};
		\draw[thick, white] (1) to [bend left=60] (4);
		\draw[thick] (3) to [bend left=30] (4);
		\draw[thick] (2) to [bend left=30] (3);
		\draw[thick] (1) to [bend left=45] (3);
		\end{tikzpicture}
	}
	+
	\vc{
		\begin{tikzpicture}[scale=0.7]
		\draw[thin, green] (0,0) to (5,0);
		\node[v] (1) at (1,0){};
		\node[v] (2) at (2,0){};
		\node[v] (3) at (3,0){};
		\node[v] (4) at (4,0){};
		\draw[thick] (1) to [bend left=60] (4);
		\draw[thick] (3) to [bend left=30] (4);
		\draw[thick] (2) to [bend left=30] (3);
		\end{tikzpicture}
	}
	\end{align*}
	\caption{\label{fig-proof}
		Graphical depiction of the calculation in the proof of Lemma~\ref{lemma-knbasis}. The two results are the same after canceling like terms.}
\end{figure}
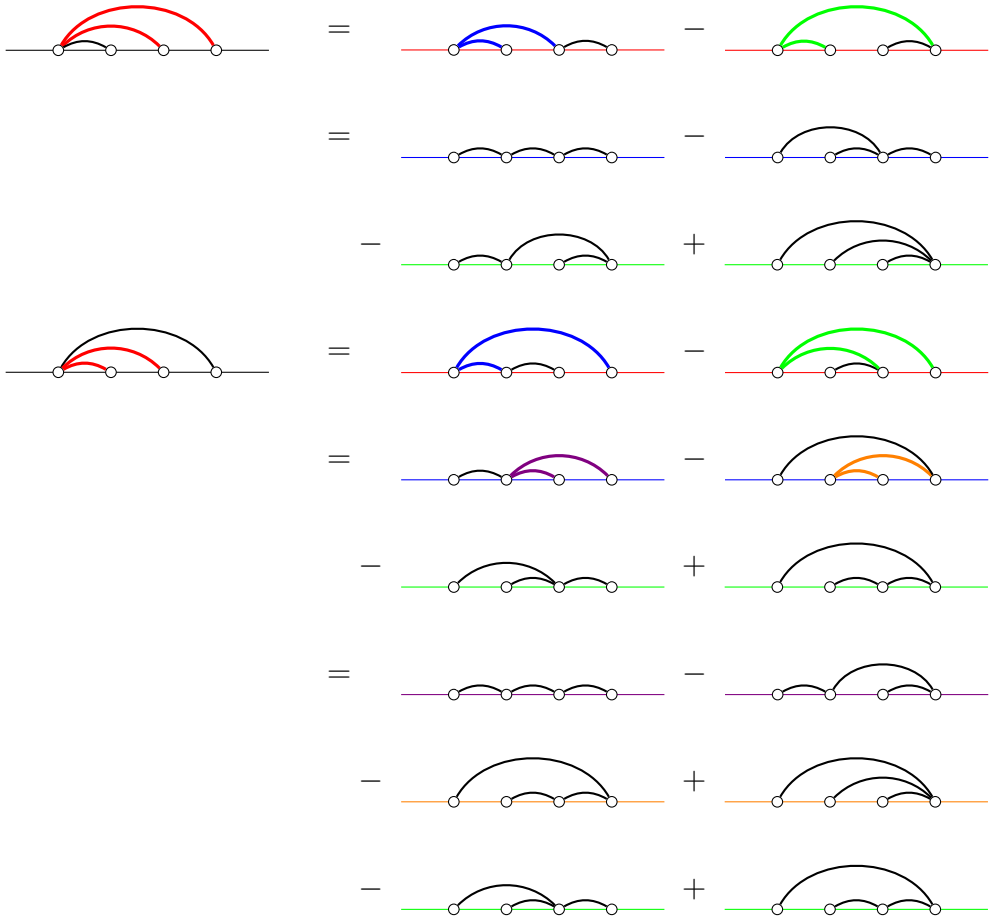

A similar calculation can be used to show that the reduction procedure will always give a unique result for any graph $G$. However, this is not enough to show that the $G$-reduced noncrossing trees give a basis, as there may a priori be other relations in $\FK_n^{\sigma}$ other than those of the form \eqref{eq-rel}. Instead, we will use the fact that $\FK_G^{\sigma} \subset \FK_n^{\sigma}$ to express $x_T$ when $T$ is a $G$-reduced tree in terms of $K_n$-reduced trees.
\subsection{Signature}

To identify noncrossing trees more easily, we associate to each one a signature.
\begin{defn}
The \emph{signature} of a noncrossing tree $T$ is the sequence $s(t) = (s_1, s_2, \dots, s_n)$ defined as follows: $s_1 = 1$; and, for $i>1$, $s_i = s_j$, where $j<i$ is minimum such that $\overline{ji} \in T$ if such a $j$ exists, otherwise $s_i=s_{i-1}+1$.
\end{defn}

For example, the signature of the noncrossing tree in Figure~\ref{fig-noncrossing} is $(1,1,2,3,2,1,2,1)$.

While the signature of a noncrossing tree is not unique, it is unique for $G$-reduced noncrossing trees.

\begin{lemma}\label{lemma-signature}
The $G$-reduced noncrossing trees have distinct signatures.
\end{lemma}
\begin{proof}
Note that $s_i=1$ if and only if the path from 1 to $i$ visits vertices in increasing order. For instance, we must have $s_n=1$ by the noncrossing condition. Moreover, if $i<j$ are consecutive vertices for which $s_i=s_j=1$, then none of the vertices between them can be connected to $s_i$, nor can they be connected to any vertices less than $i$ or greater than $j$ by the noncrossing condition. Hence the induced subgraph of $T$ on vertices $i+1, \dots, j$ is itself a noncrossing tree with signature $(s_{i+1}-1, s_{i+2}-1, \dots, s_{j-1}-1, 1)$.

We will show that the edges of $T$ connecting vertices $i$ with $s_i=1$ are completely determined by the condition of being $G$-reduced. Equivalently, we will show that there is a unique $G$-reduced tree with $s(T) = (1,1, \dots, 1)$ for any $G$. This will suffice: since the other edges lie in noncrossing subtrees with signatures as described above, these will also be determined by $s(T)$ by induction.

If $s(T) = (1, 1,\dots, 1)$, then there is exactly one edge $e_i$ of $T$ whose right endpoint is $i$ for $i>1$. We claim that there is at most one choice for $e_i$ given $e_1, \dots, e_{i-1}$. Suppose there are two edges $e_i = \overline{ji}$ and $e_i'=\overline{j'i}$ in $G$ with $j'<j$ that can be added to $e_1, \dots, e_{i-1}$ to give a noncrossing tree. Then adding $e_i'$ violates the condition of being $G$-reduced---in the definition of $G$-reduced, take $i_1=j'$, $i_{m-1}=j$, $i_m=i$, and let $i_1, i_2, \dots, i_{m-1}$ be the unique path from $j'$ to $j$ in $T$. It then follows that the only possibility for $e_i$ is when $j$ is maximum, so there is a unique $G$-reduced noncrossing tree for which $s(T)=(1,1, \dots, 1)$, as desired.
\end{proof}

We will need the following technical lemma about signatures.
\begin{lemma}\label{lemma-technical}
Let $\overline{ik} \in T$ be an edge of a noncrossing tree. Then for all $i<j<k$, $s_j \geq s_i$. If $j$ is closer to $k$ than $i$ in $T$ (that is, if the path from $j$ to $k$ in $T$ does not contain $i$), then $s_j \geq s_i+1$.
%
%
\end{lemma}
\begin{proof}
If $\overline{ik}$ is an edge, then the induced graph on $i, i+1, \dots, k$ is also a noncrossing tree. Since none of the vertices $j$ with $i<j<k$ have a neighbor to the left of $i$, we must have $s_j \geq s_i$.

Let $p$ be maximal such that  $i\leq p<k$ and $p$ is closer to $i$ than $k$ in $T$. Then by the noncrossing condition, the vertices $i, i+1, \dots, p$ are all closer to $i$ than $k$, and the vertices $p+1, \dots, k$ are all closer to $k$ than $i$, and both sets of vertices induce noncrossing trees. If $p+1 \leq j < k$, then $s_j \geq s_{p+1} = s_p+1 = s_i+1$.
\end{proof}

We now prove that the $G$-reduced noncrossing trees give a basis for all graphs $G$ by showing that the corresponding $x_T$ are linearly independent in $\FK_G^\sigma$.

\begin{thm} \label{thm-main2}
The set $\{x_T\}$, where $T$ ranges over all $G$-reduced noncrossing trees, is a basis of $\FK_G^\sigma$. Hence the relations in $\FK_G^\sigma$ are generated by those of the form \eqref{eq-rel}, and the reduction procedure always gives the unique expression of any element of $\FK_G^\sigma$ in terms of this basis.
\end{thm}
\begin{proof}
Let $B=\{x_U\}$, where $U$ ranges over all $K_n$-reduced noncrossing trees, be the basis of $\FK_n^\sigma$ as shown in Lemma~\ref{lemma-knbasis}. Write $x_U \prec x_{U'}$ if $s(U)$ is less than $s(U')$ lexicographically.

We claim that for any noncrossing tree $T$, when $x_T$ is written in terms of the basis $B$, the leading term (minimal with respect to $\prec$) will be $x_U$ with coefficient 1, where $U$ has the same signature as $T$. By Lemma~\ref{lemma-signature}, it will follow that when $T$ is a $G$-reduced noncrossing trees, the $x_T$ will have distinct leading terms, so they will be linearly independent.

Suppose we reduce $x_T$ in $\FK_n^\sigma$ using \eqref{eq-rel} for $m=3$. (By the proof of Lemma~\ref{lemma-knbasis}, these are the only relations we need to reduce $T$ to $K_n$-reduced noncrossing trees.) Write $x_T = x_{T'}-x_{T''}$, where $T' = (T \cup \overline{jk}) \backslash \overline{ik}$ and $T'' = (T \cup \overline{jk}) \backslash \overline{ij}$ for some $i<j<k$.

We first show that $s(T)=s(T')$. Clearly the first $k-1$ terms of $s(T)$ and $s(T')$ coincide. If the leftmost neighbor of $k$ in $T$ is not $i$, then this neighbor is still leftmost in $T'$, so $s_k$ does not change. If the leftmost neighbor of $k$ in $T$ is $i$, then the leftmost neighbor of $k$ in $T'$ is $j$, but since $s_i=s_j$, we have $s_i=s_j=s_k$ in both cases. It follows that $s(T)=s_(T')$.

We next show that $s(T) \prec s(T'')$. Again the first $j-1$ terms of $s(T)$ and $s(T'')$ coincide. In $T$, the leftmost neighbor of $j$ is $i$, so $s_j(T)=s_i(T)$. In $T''$, $j$ is closer to $k$ than $i$, so by Lemma~\ref{lemma-technical}, $s_j(T'') > s_i(T'') = s_i(T) = s_j(T)$. Hence $s(T) \prec s(T'')$.

Therefore at each step of the reduction, the leading term keeps the same signature and has coefficient 1. Thus once all reductions are performed, the signature of the leading term will remain the same and the coefficient will still be 1. Thus the $x_T$ are linearly independent, so they form a basis by Lemma~\ref{lemma-spanning}. This also shows that the only relations in $\FK_G^\sigma$ are the ones used in Lemma~\ref{lemma-spanning}, which are all of the form \eqref{eq-rel}.
\end{proof}

Note that Theorem~\ref{thm-main2} shows that all nontrivial relations in $\FK_G^s$ are implied by relations of the form given in Proposition~\ref{prop-cycrel}. Hence, by passing to the commutative quotient and combining these relations for all $\sigma$, this can be used to give an alternative proof of Theorem~\ref{thm-main}.

\section{Conclusion}
We have shown that $\FK_G^{ab}$ is isomorphic to the Orlik-Terao algebra $U_G$, which suggests that $\FK_G^s$ can be thought of as a noncommutative analogue of $U_G$. Since $U_G$ can be defined for any hyperplane arrangement, it may be possible to similarly describe a noncommutative analogue of $\FK_G^s$ for more general hyperplane arrangements.

The strategy of describing the monomials of $\FK_G^\sigma$ up to commutation relations in terms of noncrossing trees as in Proposition~\ref{prop-tree} can also be extended to other graded pieces of $\FK_G$; however, one must pass to noncrossing graphs on surfaces of higher genus (similar to the notion of combinatorial maps on surfaces). It has yet to be seen whether this geometric description can be helpful to derive algebraic consequences for nonsimple monomials.

When $G$ is the complete graph $K_n$ and $\sigma$ is an $n$-cycle, the dimension of $\FK_G^\sigma$ is the Catalan number $C_{n-1}$. It would be interesting to see if the dimension of $\FK_G^\sigma$ has combinatorial properties for other
graphs $G$.

\section{Acknowledgments}
The author would like to thank Karola M\'esz\'aros, John Stembridge, and Jonah Blasiak for useful conversations. This work was partially supported by an NSF Postdoctoral Research Fellowship DMS 1004375.

\bibliography{abelianization}
\bibliographystyle{plain}

\end{document}